\documentclass[oneside, 12pt]{amsart}

\usepackage[utf8]{inputenc}
\usepackage[margin=1in]{geometry}
\usepackage{amsthm}
\usepackage{enumitem}  
\usepackage[OT2,T1]{fontenc}
\usepackage{amscd, amssymb, enumitem, amsmath,mathrsfs, amsfonts, amsaddr}
\usepackage{url}
\usepackage{tikz}
\usepackage{fullpage}
\usepackage{microtype}
\usetikzlibrary{decorations.pathreplacing}
\usepackage[backref=page]{hyperref}
\usepackage{hyperref}
\usepackage[utf8]{inputenc}
\usepackage[main=english,russian]{babel}

\DeclareSymbolFont{cyrletters}{OT2}{wncyr}{m}{n}
\DeclareMathSymbol{\Sha}{\mathalpha}{cyrletters}{"58}
\newcommand{\widebar}[1]{\mkern 1.5mu\overline{\mkern-1.5mu#1\mkern-1.5mu}\mkern 1.5mu}

\usepackage[utf8]{inputenc}

\newtheorem{theorem}{Theorem}[section]
\newtheorem{lemma}[theorem]{Lemma}
\newtheorem{proposition}[theorem]{Proposition}
\newtheorem*{proposition*}{Proposition}
\newtheorem{corollary}[theorem]{Corollary}
\newtheorem*{questiona*}{Question A}
\newtheorem*{questionb*}{Question B}
\newtheorem*{theorem*}{Theorem}
\newtheorem*{question*}{Question}

\theoremstyle{definition}

\newtheorem{example}[theorem]{Example}
\newtheorem{conjecture}[theorem]{Conjecture}

\newtheorem{remark}[theorem]{Remark}

\newtheorem*{acknowledgement}{Acknowledgement}

\theoremstyle{remark}

\title{An analogue of a conjecture of Rasmussen and Tamagawa for abelian varieties over function fields}
\author{Mentzelos Melistas}
\address{Charles University, Faculty of Mathematics and Physics, Department of
Algebra, Sokolov\-sk\' a 83, 18600 Praha~8, Czech Republic}

\address{University of Twente, Department of Applied Mathematics, Drienerlolaan 5, 7522 NB Enschede, The Netherlands}

\date{\today}

\begin{document}

\maketitle

\begin{abstract}
   Let $L$ be a number field and let $\ell$ be a prime number. Rasmussen and Tamagawa conjectured, in a precise sense, that abelian varieties whose field of definition of the $\ell$-power torsion is both a pro-$\ell$ extension of $L(\mu_\ell)$ and unramified away from $\ell$ are quite rare. In this paper, we formulate an analogue of the Rasmussen--Tamagawa conjecture for non-isotrivial abelian varieties defined over function fields. We provide a proof of our analogue in the case of elliptic curves. In higher dimensions, when the base field is a subfield of the complex numbers, we show that our conjecture is a consequence of the uniform geometric torsion conjecture. Finally, using a theorem of Bakker and Tsimerman we also prove our conjecture unconditionally for abelian varieties with real multiplication.
\end{abstract}

\section{introduction}

Let $L$ be a number field and let $\ell$ be a prime number. Denote by $\Tilde{L_{\ell}}$ the maximal pro-$\ell$ extension of $L(\mu_\ell)$ which is unramified away from $\ell$, where $\mu_{\ell}$ is the group of $\ell$-th roots of unity in a fixed algebraic closure $\widebar{L}$ of $L$. Given an integer $d \geq 1$, a prime $\ell$, and a number field $L$, we also denote by $\mathscr{A}(L,d, \ell)$ the set of $L$-isomorphism classes of $d$-dimensional abelian varieties $A/L$ which satisfy the following inclusion $$L(A[\ell^{\infty}]) \subseteq\Tilde{L_{\ell}}.$$

The motivation to consider the fields $\Tilde{L_{\ell}}$ comes from Galois representations arising from $\mathbb{P}_{01\infty}^1:=\mathbb{P}_{\widebar{\mathbb{Q}}}^1 \setminus \{ 0, 1, \infty \}$. More precisely, let $\Phi: \text{Gal}(\widebar{L}/L) \longrightarrow \text{Out}(\pi_1^{\ell}(\mathbb{P}_{01\infty}^1))$ be the natural outer Galois representation of the pro-$\ell$ fundamental group of $\mathbb{P}_{01\infty}^1$ and let $M:=\widebar{L}^{\text{ker}(\Phi)}$ be the fixed field of its kernel. Anderson and Ihara in \cite{andersonihara} showed the inclusion $M \subseteq \Tilde{L_{\ell}}$. The following question, posed by Ihara in \cite{ihara}, is still open; For $L=\mathbb{Q}$, is it true that $M=\Tilde{L_{\ell}}$? A natural source of sub-extensions of $M$ comes from abelian varieties that belong to $\mathscr{A}(L,d, \ell)$. Examples of primes $\ell$ and abelian varieties in $\mathscr{A}(L,d, \ell)$ have been found in \cite{andersonihara}, \cite{r04}, \cite{pr07}, and \cite{rt08}. However, it seems that such examples are actually quite rare. This motivated Rasmussen and Tamagawa, in \cite{rt08}, to formulate the following conjecture.

\begin{conjecture}\label{conjecture1}(Rasmussen--Tamagawa)
For any number field $L$ and integer $d \geq 1$ there exists a number $N=N(L,d)$ such that $\mathscr{A}(L,d, \ell)= \emptyset $ for $\ell >N$.
\end{conjecture}

For an integer $d \geq 1$ and a number field $L$, we define $$\mathscr{A}(L,d)= \{ ([A], \ell) \; : \; A \in \mathscr{A}(L,d, \ell) \}.$$ We note that it follows from the Shafarevich conjecture, which is a theorem due to Faltings \cite{faltingsisogeny}, that for fixed $L$, $d$, and $\ell$ as above the set $\mathscr{A}(L,d, \ell)$ is finite. Therefore, Conjecture \ref{conjecture1} is equivalent to the following statement; For any number field $L$ and integer $d \geq 1$ the set $\mathscr{A}(L,d)$ is finite.

Conjecture \ref{conjecture1} has been proven in many cases. First, in \cite{rt08}, Conjecture \ref{conjecture1} is proved in the case where $d=1$ and $L=\mathbb{Q}$, as well as in the case where $d=1$ and $L/\mathbb{Q}$ is a quadratic extension such that $K$ is not an imaginary quadratic field of class number $1$. Moreover, Rasmussen and Tamagawa \cite{rt17} proved Conjecture \ref{conjecture1} conditionally on the Generalized Riemann Hypothesis, unconditionally for all semi-stable abelian varieties, and unconditionally for $d=1$ when $L/\mathbb{Q}$ has degree $2$, $3$, or when $L/\mathbb{Q}$ is Galois of exponent $3$. Bourdon \cite{bourdon15}, for CM elliptic curves, and more generally Lombardo \cite{lombardo18}, for CM abelian varieties, proved a uniform version of Conjecture \ref{conjecture1}. Finally, a Drinfeld module analogue has been proven by Okumura \cite{okumura}.

In this paper we are interested in formulating an analogue of Conjecture \ref{conjecture1} for abelian varieties over function fields. More precisely let $k$ be a perfect field of characteristic $p\geq 0$ and let $C/k$ be a smooth, projective, and geometrically connected curve over $k$. We denote by $K$ the function field $k(C)$ of the curve $C/k$. We are interested in stating an analogue of Conjecture \ref{conjecture1} over the field $K$. 

Fix an algebraic closure $\widebar{K}$ of $K$. An abelian variety $A/K$ is called {\it isotrivial} if the base change $A_{\widebar{K}}/\widebar{K}$ of $A/K$ to $\widebar{K}$ is {\it constant}, i.e., if there exists an abelian variety $A_0/\widebar{k}$ such that $A_{\widebar{K}} \cong A_0 \times_k \widebar{K}$. Finally, an abelian variety $A/K$ is called {\it non-isotrivial} if it is not isotrivial.

 If $\ell$ is any prime with $\ell \neq p$, then we let $\Tilde{K_{\ell}'}$ be the maximal pro-$\ell$ extension of $K(\mu_\ell)$, where $\mu_{\ell}$ is the group of $\ell$-th roots of unity in a fixed algebraic closure $\widebar{K}$ of $K$. Moreover, we denote by $\mathscr{A}'(K,d, \ell)$ the set of $K$-isomorphism classes of $d$-dimensional non-isotrivial abelian varieties $A/K$ which satisfy the following condition $$K(A[\ell^{\infty}]) \subseteq\Tilde{K_{\ell}'}.$$ 

Our analogue of Conjecture \ref{conjecture1} is the following.

\begin{conjecture}\label{conjecturefunctionfields}
Let $k$ be a perfect field and let $C/k$ be a smooth, projective, and geometrically connected curve with function field $K=k(C)$. Then for any integer $d \geq 1$ there exists a number $N=N(K,d)$ such that $\mathscr{A}'(K,d, \ell) = \emptyset$ for $\ell > N$. 
\end{conjecture}

We note in contrast to the number field case for each prime $\ell$ the set $\mathscr{A}'(K,d, \ell)$ is not necessarily finite. This is because we do not impose any ramification requirements in the definition of $\Tilde{K_{\ell}'}$ as in the definition of $\Tilde{L_{\ell}}$. For every prime $\ell$, which is different from the characteristic of $K$, one could choose a place $\mathfrak{P}_{\ell}$ of $K$ and require that $\Tilde{K_{\ell}'}$ is in addition unramified away from $\mathfrak{P}_{\ell}$. We do not require this ramification condition in our case because in our theorem below such a restrictive condition is not necessary.

Our main theorem in this article is the following theorem, which provides proof for Conjecture \ref{conjecturefunctionfields} in dimension $1$.

\begin{theorem}\label{maintheorem}
Let $k$ be a perfect field and let $C/k$ be a smooth, projective, and geometrically connected curve of genus $g$ with function field $K=k(C)$. Then we have that $\mathscr{A}'(K,1, \ell) = \emptyset$ for every $\ell > 6 +\sqrt{1+24g}$. 
\end{theorem}

We also produce examples (Example \ref{examplegenus0} and Example \ref{examplegenus1}) which show that Theorem \ref{maintheorem} is sharp when $g=0$ and $g=1$. In higher dimensions, we are not able to prove Conjecture \ref{conjecturefunctionfields} in full generality. However, when the characteristic of $k$ is zero, we provide evidence for our conjecture by showing that the uniform geometric torsion conjecture for abelian varieties over function fields (Conjecture \ref{uniformgeometrictorsion}) implies Conjecture \ref{conjecturefunctionfields}, see Theorem \ref{torsionconjectureimpliesconjecture} below. We also prove the following unconditional theorem.

\begin{theorem}\label{maintheorem2}
    Let $d$ be a positive integer, let $k$ be a field, and let $C/k$ be a smooth, projective, and geometrically connected curve of genus $g$ with function field $K=k(C)$. Then the following are true.
    \begin{enumerate}
        \item (Proposition \ref{propositionsemistable})  If $k$ is a finite field of characteristic $p$ and $A/K$ is an abelian variety that belongs to $\mathscr{A}'(K,d, \ell)$ for some $\ell > 2d+1$, then $A/K$ has semi-stable reduction at all places of $K$.
        \item (Theorem \ref{theoremrm}) Assume that $k$ is a subfield of $\mathbb{C}$ and let $\mathscr{A}_{\text{RM}}'(K,d, \ell)$ be the set of abelian varieties $A/K$ that belong to $\mathscr{A}'(K,d, \ell)$ and have maximal real multiplication. Then  there exists a constant $N:=N( g, d)$ (depending only on $g$ and $d$) such that if $A/K$ an abelian variety that belongs to $\mathscr{A}_{\text{RM}}'(K,d, \ell)$, then we must have that $\ell < N$.
    \end{enumerate}
\end{theorem}

This article is organized as follows. In Section \ref{section2} we prove Theorem \ref{maintheorem} and we show that our theorem is sharp, when the genus of $C/k$ is $0$ or $1$. Moreover, we prove part $(ii)$ of Theorem \ref{maintheorem2} and we show that the uniform geometric torsion conjecture over function fields (Conjecture \ref{uniformgeometrictorsion}) implies Conjecture \ref{conjecturefunctionfields}. Section \ref{section3} mostly concerns the case where the characteristic of $k$ is positive. After proving part $(i)$ of Theorem \ref{maintheorem2}, we prove that Conjecture \ref{conjecturefunctionfields} is true for a general class of Jacobians, and then we discuss why an analogue of Conjecture \ref{conjecturefunctionfields} for abelian varieties over finite fields does not seem to exist. Finally, in the last section, we consider the possibility of formulating a generalization of Conjecture \ref{conjecturefunctionfields} by replacing the fields $K(\mu_{\ell})$ with more general field extensions.

\begin{acknowledgement}
 This work was supported by Czech Science Foundation (GA\v CR) [21-00420M] and by Charles University Research Center program [UNCE/SCI/022]. The author would like to thank Dino Lorenzini for some useful comments on an earlier version of this manuscript. I would also like to thank the anonymous referee for providing several very useful suggestions and corrections that improved this article.
\end{acknowledgement}

\section{Proof of Theorem \ref{maintheorem}}\label{section2}

In this section, we first prove Theorem \ref{maintheorem}. Then we relate Conjecture \ref{conjecturefunctionfields} to the uniform geometric torsion conjecture over function fields. More precisely, we show that Conjecture \ref{uniformgeometrictorsion} implies Conjecture \ref{conjecturefunctionfields}. Finally, we prove part $(ii)$ of Theorem \ref{maintheorem2} using a Theorem of Bakker and Tsimerman from \cite{bakkertsimerman}.

We start our discussion with a lemma that will be very useful for our results.

\begin{lemma}\label{lemmatorsionpoint}
    Let $k$ be a perfect field of characteristic $p \geq 0$ and let $C/k$ be a smooth, projective, and geometrically connected curve of genus $g$ with function field $K=k(C)$. Let $\ell \neq p$ be a prime. If the degree of the extension $K(A[\ell])/K(\mu_{\ell})$ is equal to a power of $\ell$, then $A_{K(\mu_{\ell})}/K(\mu_{\ell})$ has a $K(\mu_{\ell})$-rational point of order $\ell$. 
\end{lemma}

\begin{proof}

    Let $A/K$ be an abelian variety such that $[K(A[\ell]):K(\mu_{\ell})]$ is equal to a power of $\ell$. We need to show that $A_{K(\mu_{\ell})}/K(\mu_{\ell})$ has a $K(\mu_{\ell})$-rational point of order $\ell$. It suffices to show that the action of the $\ell$-group $\text{Gal}(K(A[\ell])/K(\mu_{\ell}))$ on $A[\ell]$ has a non-trivial fixed point.This holds since any $\ell$-group acting on a nontrivial finite-dimensional $\mathbb{F}_\ell$-vector space fixes at least one non-zero vector.  
\end{proof}

In what follows when we refer to the genus of a function field $K$ we will mean the genus of $C/k$, where $C/k$ is a smooth, projective, and geometrically connected curve with function field $K$. We are now ready to proceed to the proof of Theorem \ref{maintheorem}. The key to our proof, and to the proofs of Theorems \ref{torsionconjectureimpliesconjecture} and \ref{theoremrm} below, is that $K(\mu_{\ell})/K$ is a constant extension and, hence, $K(\mu_{\ell})$ has the same genus as $K$.

\begin{proof}[{\it Proof of Theorem \ref{maintheorem}}]
     Let $E/K$ be a non-isotrivial abelian variety that belongs to $\mathscr{A}'(K,1, \ell)$. We will show that $\ell \leq 6 + \sqrt{1+24g}$. First, since $K(\mu_{\ell})/K$ is a constant extension, by our assumptions on $C/k$, it follows that the genus of $K(\mu_{\ell})$ is $g$. Moreover, using Lemma \ref{lemmatorsionpoint} we find that $E_{K(\mu_{\ell})}/K(\mu_{\ell})$ has a $K(\mu_{\ell})$-rational point of order $\ell$. Therefore, using work of Levin (see \cite[Page 460]{levin}) we obtain that $\ell \leq 6 + \sqrt{1+24g}$. This completes the proof of our theorem.
\end{proof}

The following conjecture, see \cite[Page 228]{cadorettamagawaweakvariant}, is called the uniform geometric torsion conjecture for abelian varieties over function fields.

\begin{conjecture}\label{uniformgeometrictorsion}
    Let $k$ be an algebraically closed field of characteristic $0$ and let $C/k$ be a smooth, projective, and geometrically connected curve of genus $g$ with function field $K=k(C)$. Then for any integer $d$, there exists a constant $N:=N(k, g, d)$ (which depends on $k$, $g$, and $d$) such that for any $d$-dimensional abelian variety $A/K$ containing no nontrivial isotrivial abelian subvariety the torsion subgroup $A(K)_{\text{tors}}$ is contained in $A[N]$.
\end{conjecture}

Following the same strategy as in the proof of Theorem \ref{maintheorem}, we can prove the following theorem.

\begin{theorem}\label{torsionconjectureimpliesconjecture}
    Let $k$ be a field of characteristic $0$ and let $C/k$ be a smooth, projective, and geometrically connected curve of genus $g$ with function field $K=k(C)$. Then Conjecture \ref{uniformgeometrictorsion} (over $\widebar{k}$) implies Conjecture \ref{conjecturefunctionfields}.
\end{theorem}
\begin{proof}
    Let $A/K$ be a non-isotrivial abelian variety that belongs to $\mathscr{A}'(K,d, \ell)$ and we will show that $\ell$ is bounded. Using Lemma \ref{lemmatorsionpoint} we find that $A_{K(\mu_{\ell})}/K(\mu_{\ell})$ has a $K(\mu_{\ell})$-rational point of order $\ell$. Let now $\widebar{k}$ be the algebraic closure of $k$ and let $K_{\widebar{k}}$ denote the function field of $C_{\widebar{k}}/\widebar{k}$, which has genus $g$. It follows that $A_{K_{\widebar{k}}}/K_{\widebar{k}}$ has a $K_{\widebar{k}}$-rational point of order $\ell$. Therefore, using Conjecture \ref{uniformgeometrictorsion} we obtain that $\ell \leq N$, for some constant $N$ that depends on $k$, $g$, and $d$. This proves our theorem.
\end{proof}

Recall that an abelian variety with maximal real multiplication is an abelian variety $A/K$ of dimension $d$ such that there exists an injection $\mathcal{O}_F \longrightarrow \text{End}(A)$, where $\mathcal{O}_F$ is the ring of integers of a totally real number field $F$ with $[F:\mathbb{Q}]=d$ and $\text{End}(A)$ denotes the ring of endomorphisms of $A/K$ which are defined over $K$. The following theorem is due to Bakker and Tsimerman, see \cite[Corollary B]{bakkertsimerman}.

\begin{theorem}\label{bakkertsimermanrm} 
    Fix positive  integers $d$,$g$, let $k=\mathbb{C}$, and let $K$ be the function field of a smooth, projective, and geometrically connected curve $C/k$ of genus $g$. Then there exists a constant $N:=N( g, d)$ such that for any $d$-dimensional non-isotrivial abelian variety $A/K$ with maximal real multiplication the torsion subgroup $A(K)_{\text{tors}}$ is contained in $A[N]$. 
\end{theorem}

Recall that $\mathscr{A}_{\text{RM}}'(K,d, \ell)$ is the set of abelian varieties $A/K$ that belong to $\mathscr{A}'(K,d, \ell)$ and have maximal real multiplication.
\begin{theorem}\label{theoremrm}
    Let $k$ be a subfield of $\mathbb{C}$ and let $C/k$ be a smooth, projective, and geometrically connected curve of genus $g$ with function field $K=k(C)$. Then  there exists a constant $N:=N( g, d)$ such that if $A/K$ a $d$-dimensional abelian variety that belongs to $\mathscr{A}_{\text{RM}}'(K,d, \ell)$, then we must have that $\ell \leq N$.
\end{theorem}
\begin{proof}
    The proof is similar to the proof of Theorem \ref{torsionconjectureimpliesconjecture}. Assume that there exists a non-isotrivial abelian variety $A/K$  that belongs to $\mathscr{A}_{\text{RM}}'(K,d, \ell)$ and we will show that $\ell \leq N$ for some $N$ that depends on $g$ and $d$. Using Lemma \ref{lemmatorsionpoint} we find that $A_{K(\mu_{\ell})}/K(\mu_{\ell})$ has a $K(\mu_{\ell})$-rational point of order $\ell$. Let now $K_{\mathbb{C}}$ denote the function field of $C_{\mathbb{C}}/\mathbb{C}$, which still has genus $g$. It follows that $A_{K_{\mathbb{C}}}/K_{\mathbb{C}}$ has a $K_{\mathbb{C}}$-rational point of order $\ell$. Therefore, using Theorem \ref{bakkertsimermanrm} we obtain that $\ell \leq N$, for some constant $N$ that depends on $g$ and $d$. This proves our theorem.
\end{proof}

The following proposition can be thought of as the converse of Lemma \ref{mainlemma} for elliptic curves. We will use this proposition below to find elements of $\mathscr{A}'(K,1, \ell)$ and show that the bound provided by Theorem \ref{maintheorem} is sharp when the genus of $K$ is $0$ or $1$.

\begin{proposition}\label{proponepoint} Let $K$ be a field of characteristic $p\geq 0$, let $\ell$ be prime with $\ell \neq p$, and let $E/K$ be an elliptic curve. If the base change  $E_{K(\mu_{\ell})}/K(\mu_{\ell})$ of $E/K$ to $K(\mu_{\ell})$ has a $K(\mu_{\ell})$-rational point of order $\ell$, then $K(E[\ell^{\infty}]) \subseteq\Tilde{K_{\ell}'}.$

\end{proposition}
\begin{proof}

 Assume that $E_{K(\mu_{\ell})}/K(\mu_{\ell})$ has a $K(\mu_{\ell})$-rational point $P$ of order $\ell$. We first show, using an idea from \cite[Lemma 3]{bourdon15}, that $[K(E[\ell]): K(\mu_{\ell})]$ is either $1$ or $\ell$. By choosing a basis $\{P, Q\}$ of $E[\ell]$ and considering the mod-$\ell$ Galois representation with respect to this basis we find that the Galois group $\text{Gal}(K(E[\ell])/K(\mu_{\ell}))$ is isomorphic to the group $H$ generated by the matrix $\begin{pmatrix}
    1 & a \\
    0 & 1
\end{pmatrix}$ inside the group $\text{GL}_2(\mathbb{Z}/\ell \mathbb{Z})$, for some $a \in \mathbb{F}_{\ell}$. On the other hand, the group $H$ has cardinality either $1$ or $\ell$. Therefore, the degree $[K(E[\ell]): K(\mu_{\ell})]$ is either $1$ or $\ell$.

In order to complete our proof it is enough to show that $K(E[\ell^n])/K(E[\ell])$ is an extension of degree a power of $\ell$, for every integer $n \geq 2$. We now show this. Recall that for every $n$ we have a natural injective group homomorphism $$\text{Gal}(K(E[\ell^{n+1}])/K(E[\ell^n])) \longrightarrow \text{Aut} (E[\ell^{n+1}]/E[\ell^n]), $$ where the group on the right is the group of automorphisms of $E[\ell^{n+1}]$ which fix $E[\ell^{n}]$ pointwise. Therefore, the order of the group $\text{Gal}(K(E[\ell^{n+1}])/K(E[\ell^n]))$ divides the order of the group $\text{Aut} (E[\ell^{n+1}]/E[\ell^n])$. Since $\ell \neq p$, we know that $E[\ell^n] \cong \mathbb{Z}/\ell^n \mathbb{Z} \times \mathbb{Z}/\ell^n \mathbb{Z}$ and, hence, the group $\text{Aut} (E[\ell^{n+1}]/E[\ell^n])$ has order a power of $\ell$. This proves that $K(E[\ell^{n+1}])/K(E[\ell^n])$ is an extension of degree equal to a power of $\ell$. This proves our proposition.
\end{proof}

 Let $k$ be a perfect field of characteristic $p \geq 0$. We now show that the bound provided by Theorem \ref{maintheorem} is sharp when $K= k(C)$, where $C/k$ is a smooth, projective, and geometrically connected curve of genus either $0$ or $1$. Let $\ell_0 $ be the largest prime such that $\mathscr{A}'(K,1, \ell_0) \neq \emptyset$. One may then wonder how large $\ell_0$ can be. It follows from Theorem \ref{maintheorem} that $\ell_0 \leq 7 $ if $g=0$ and that $\ell_0 \leq 11 $ if $g=1$. Using Proposition \ref{proponepoint} we find that if $E/K$ is an elliptic curve with a $K$-rational point of order $\ell \neq p$, then $E/K$ belongs to $\mathscr{A}'(K,1, \ell)$. Therefore, in order to show that Theorem \ref{maintheorem} is sharp when $g=0$ or $1$, then we just need to find examples of elliptic curves with a $K$-rational point of order $7$ and $11$, respectively. We do so in the following two examples, using (well-known) explicit equations for modular curves.

 \begin{example}\label{examplegenus0}
     Let $k$ be a perfect field of characteristic $p \neq 7$, let $K=k(t)$, and let $f \in K$ be a non-constant rational function. Consider the elliptic curve $E_f/K$ given by the following Weierstrass equation $$E_f: \; y^2+(1-f(f-1))xy-f^2(f-1)y=x^3-f^2(f-1)x^2.$$ The discriminant of the above equation is $$\Delta_f=f^7(f-1)^7(f^3-8f^2+5f+1).$$ Since $\Delta_f$ is non-zero in $K$, we see that indeed $E_f/K$ is an elliptic curve. The $c_4$-invariant of $E_f/K$ is $$c_4=(f^2-f+1)(f^6-11f^5+30f^4-15f^3-10f^2+5f+1).$$ Moreover, the curve $E_f/K$ is non-isotrvial because the $j$-invariant of $E_f/K$ is $j_{E_f}=\frac{c_4^3}{\Delta_f}$ is non-constant. Finally, since the point $(0,0)$ is a $K$-rational point of order $7$ of $E_f/K$ (see \cite[Section 4.4]{hus}), it follows from Proposition \ref{proponepoint} that the curve $E_f/K$ belongs to $\mathscr{A}'(K,1, 7)$. 
 \end{example}

\begin{example}\label{examplegenus1}
     Let $k$ be a perfect field of characteristic $p \neq 11$. Consider the modular curve $X_1(11)/k$ which parametrizes elliptic curves with torsion points of order $11$. The curve $X_1(11)/k$ is an elliptic curve which can be given by the following affine equation $$u^2+(t^2+1)u+t=0.$$ We note that the above equation is not the standard short Weierstrass equation that one can find in the LMFDB database \cite{lmfdb}, but it is an equation optimized for computations \cite{sutheroptimized}. Another equation for $X_1(11)/k$ can be found in \cite[Example III.1.1.3]{silverman2}. Let $K=k(X_1(11))=k(t,u)$ and consider, for every $n \geq 0$, the elliptic curve $E_n/K$ given by the following Weierstrass equation
     $$E_n: \; y^2+(s-rs+1)^{p^n}xy+(rs-r^2s)^{p^n}y=x^3+(rs-r^2s)^{p^n}x^2,$$
     where $r=ut+1$ and $s=-t+1.$ For every $n\geq 0$ it follows that $E_n/K$ has a $K$-rational point of order $11$, namely $(0,0)$. It follows from Proposition \ref{proponepoint} that  $E_n/K$ belongs to $\mathscr{A}'(K,1, 11)$. One can check that $E_n/K$ is non-isotrivial by computing the $j$-invariant of $E_n/K$ and checking that it is non-constant (this also follows from the general theory because $E_0/K$ comes from the universal elliptic curve with a point of order $11$ and $E_{n+1}/K$ is the base change of $E_n/K$ by the absolute Frobenius morphism of $K$). For more information on explicit equations of elliptic curves with torsion points over function fields see \cite{sutheroptimized}, \cite{sutherland} and/or \cite[Page 64]{mcdonaldthesis}. We note though that there is a typo in the equation at the bottom of Page 64 of \cite{mcdonaldthesis}.
\end{example}

\section{abelian varieties over fields of positive characteristic}\label{section3}

In this section, we first show that if an abelian variety $A/K$ belongs to $\mathscr{A}'(K,d, \ell)$ for $\ell>2d+1$, then it must have semi-stable reduction at all places of $K$. We then prove that Conjecture \ref{conjecturefunctionfields} is true for a general class of Jacobians. Finally, we consider the situation in the case where $K$ is a finite field, instead of a function field.

Before we proceed we need to recall a few basic facts concerning reduction of abelian varieties. Let $\mathcal{O}_{K_v}$ be a discrete valuation ring with valuation $v$, fraction field $K_v$, and perfect residue field $k_v$. Let $A/K_v$ be an abelian variety of dimension $d$ with N\'eron model $\mathcal{A}/\mathcal{O}_{K_v}$ (see \cite{neronmodelsbook} or \cite{neronmodelslorenzini} for the definition as well as the basic properties of N\'eron models). The special fiber $\mathcal{A}_{k_v}/k_v$ of $\mathcal{A}/\mathcal{O}_{K_v}$ is a smooth commutative group scheme. We denote by $\mathcal{A}^0_{k_v}/k_v$ the connected component of the identity of $\mathcal{A}_{k_v}/k_v$. By a theorem of Chevalley (see \cite[Theorem 1.1]{con}) we have a short exact sequence $$0\longrightarrow T \times U \longrightarrow \mathcal{A}^0_{k_v} \longrightarrow B \longrightarrow 0, $$
where $T/k_v$ is a torus, $U/k_v$ is a unipotent group, and $B/k_v$ is an abelian variety. We say that $A/K_v$ has semi-stable reduction if $\text{dim}(U)=0$. If now $k$ is a finite field of characteristic $p$ and $K=k(C)$, where $C/k$ is a smooth, projective, and geometrically connected curve, then we will say that $A/K$ has semi-stable reduction at a place $v$ if $A_{K_v}/K_v$ has semi-stable reduction. Here $K_v$ is the completion of $K$ at $v$. Moreover, we will say that $A/K$ is semi-stable if it has semi-stable reduction at every place $v$ of $K$.

\begin{proposition}\label{propositionsemistable}
    Let $k$ be a finite field of characteristic $p$ and let $C/k$ be a smooth, projective, and geometrically connected curve with function field $K=k(C)$. If $A/K$ is an abelian variety that belongs to $\mathscr{A}'(K,d, \ell)$ for some $\ell > 2d+1$, then $A/K$ is semi-stable.
\end{proposition}
\begin{proof}
     Let $A/K$ be an abelian variety that belongs to $\mathscr{A}'(K,d, \ell)$, for some $\ell > 2d+1$, and we will show that it has semi-stable reduction at every finite place $v$ of $K$. Since $K(\mu_{\ell})/K$ is a constant extension, then it is an everywhere unramified extension. Moreover, since an everywhere unramified base extension does not affect whether an abelian variety is semi-stable, by considering the base change $A_{K(\mu_{\ell})}/K(\mu_{\ell})$ of $A/K$ to $K(\mu_{\ell})$, we can assume from now on that $K(\mu_{\ell})=K$ and that $K(A[\ell])/K$ is a field extension of degree equal to a power of $\ell$.
     
     Let $v$ be a place of $K$ and let $K_v$ be the completion of $K$ at $v$. We also denote by $K_v^{unr}$ the maximal unramified extension of $K_v$. Since $K(A[\ell])/K$ is a Galois extension, we find that the degree of the extension $K_v^{unr}(A[\ell])/K_v^{unr}$ divides the degree of the extension $K(A[\ell])/K$. Therefore, since $K(A[\ell])/K$ is a field extension of degree equal to a power of $\ell$, we find that the same is true for $K_v^{unr}(A[\ell])/K_v^{unr}$. It follows from a theorem due to Raynaud \cite[Proposition 4.7]{sga7I} that the base change $A_{K_v^{unr}(A[\ell])}/K_v^{unr}(A[\ell])$ of $A/K$ to $K_v^{unr}(A[\ell])$ has semi-stable reduction. On the other hand, by \cite[Theorem 3.8]{neronmodelslorenzini} there exists a minimal extension $K_{A_{K_v^{unr}}}/K_v^{unr}$ over which $A_{K_v^{unr}}/K_v^{unr}$ acquires semi-stable reduction. Moreover, every prime that divides the degree of $K_{A_{K_v^{unr}}}/K_v^{unr}$ is at most $2d+1$, see \cite[Theorem 6.8]{conradsemistablereduction}. By the minimality of $K_{A_{K_v^{unr}}}/K_v^{unr}$ we must have that $$K_{A_{K_v^{unr}}} \subseteq K_v^{unr}(A[\ell]). $$ By comparing the degrees of the extensions $K_{A_{K_v^{unr}}}/K_v^{unr}$ and $K_v^{unr}(A[\ell])/K_v^{unr}$ we find that $K_{A_{K_v^{unr}}} =K_v^{unr}$. This proves that the variety $A_K/K$ has semi-stable reduction at $v$.
     \end{proof}

    \begin{remark}
        Let $L$ be a number field and let $A/L$ be an abelian variety that belongs to the set $\mathscr{A}(L,g, \ell)$, for some prime number $\ell$. If $L(\mu_{\ell})=L$ and $\ell > 2d+1$, then exactly the same argument as in the proof of Proposition \ref{propositionsemistable} proves that $A/L$ has everywhere semi-stable reduction.
    \end{remark}

 We now prove for a general class of Jacobians over $\mathbb{F}_q(t)$ that if they belong to $\mathscr{A}'(\mathbb{F}_q(t), d, \ell)$, then $\ell=2$, thus providing additional evidence for Conjecture \ref{conjecturefunctionfields} over $\mathbb{F}_q(t)$. Let $d$ be a positive integer, let $q$ be a odd prime power and let $f(x) \in \mathbb{F}_q[x]$ be a monic square-free polynomial of degree $2d$. Consider the hyperelliptic curve $C/\mathbb{F}_q(t)$ given by the following affine equation $$C: y^2=(t-x)f(x)$$ and denote by $J_C/\mathbb{F}_q(t)$ the Jacobian of the curve $C/\mathbb{F}_q(t)$, which is a $d$-dimensional abelian variety. The following proposition is a consequence of a Theorem of Hall \cite{hall}, originally due to Jiu-Kang Yu.

\begin{proposition}
If the abelian variety $J_C/\mathbb{F}_q(t)$ belongs to $\mathscr{A}'(\mathbb{F}_q(t), d, \ell)$, then $\ell =2$.
\end{proposition}
\begin{proof}
    Assume that $\ell >2$. It follows from \cite[Theorem 4.1]{hall} that the extension $K(J_C[\ell])/K(\mu_\ell)$ has Galois group $\text{Sp}(2g, \mathbb{F}_{\ell})$. On the other hand, we know that the group $\text{Sp}(2d, \mathbb{F}_{\ell})$ has order equal to $\ell^{d^2} \prod_{i=1}^d (\ell^{2i}-1).$ Therefore, the field extension $K(J_C[\ell])/K(\mu_\ell)$ cannot have order a power of $\ell$ and, hence, the variety $J_C/K$ does not belong to $\mathscr{A}'(\mathbb{F}_q(t), d, \ell)$. This proves our proposition.
\end{proof}

We end this section by discussing why a similar analogue of the Rasmussen Tamagawa conjecture does not seem to exist for abelian varieties over finite fields. Let $\mathbb{F}_q$ be a finite field with $q$ elements, where $q$ is a power of a prime $p$. We fix an algebraic closure $\widebar{\mathbb{F}_q}$ of $\mathbb{F}_q$ and for every integer $m \geq 1$ we denote by $\mathbb{F}_{q^m}$ the unique subfield of $\widebar{\mathbb{F}_q}$ that has $q^m$ elements. 

We first recall some necessary background on abelian varieties over finite fields. The reader is referred to \cite[Chapter II]{milneabelianvarieties} for more information. Let $A/\mathbb{F}_q$ be a $d$-dimensional abelian variety and denote by $\phi$ the Frobenius endomorphism of $A/\mathbb{F}_q$. Recall that the characteristic polynomial of $\phi$, denoted by $P_{\phi}(x)$, is a monic polynomial of degree $2d$ which belongs to $\mathbb{Z}[x]$. Write $P_{\phi}(x)=\prod_{i=1}^{2d} (x-\alpha_i)$, where $\alpha_1,...,\alpha_{2d} \in \mathbb{C}$ are the roots of $P_{\phi}(x)$ (not necessarily distinct). For every integer $m \geq 1$ we know the following equality 
\begin{align}\label{equationpointsfinitefields}
    \# A(\mathbb{F}_{q^m})=\prod_{i=1}^{2d} (1-\alpha_i^m).
\end{align}

If $A/\mathbb{F}_q$ is an abelian variety and $\ell \neq p$, then, using an argument similar to the second paragraph of the proof of Proposition \ref{proponepoint}, we find that $\mathbb{F}_q(A[\ell^{\infty}])$ is contained in the maximal pro-$\ell$ extension of $\mathbb{F}_q$ if and only if the field extension $\mathbb{F}_q(A[\ell])/\mathbb{F}_q(\mu_{\ell})$ has degree a power of $\ell$. Let now $f$ be the smallest integer such that $q^f \equiv 1 \: (\text{mod } \ell ) $. The degree of the extension $\mathbb{F}_q(\mu_{\ell})/\mathbb{F}_q$ is equal to $f$ and, hence, we have that $\mathbb{F}_q(\mu_{\ell})= \mathbb{F}_{q^f}$. We now want to examine when $[\mathbb{F}_q(A[\ell]):\mathbb{F}_q(\mu_{\ell})]=\ell^e$, for some $e \geq 0$, i.e., when $\mathbb{F}_q(A[\ell])= \mathbb{F}_{q^{f\ell^e}}$. Using Equation $($\ref{equationpointsfinitefields}$)$, we find that $$A(\mathbb{F}_{q^{f\ell^e}})=\prod_{i=1}^{2d} (1-\alpha_i^{f\ell^e}).$$ Since $\ell^2$ divides $\# A(\mathbb{F}_q(A[\ell]))$, if $ A(\mathbb{F}_q(A[\ell]))= A(\mathbb{F}_{q^{f\ell^e}}) $ for some $e \geq 0$, then we find that $\ell^2$ must divide the product $\prod_{i=1}^{2d} (1-\alpha_i^{f\ell^e})$. Therefore, whether $\mathbb{F}_q(A[\ell^{\infty}])$ is contained in the maximal pro-$\ell$ extension of $\mathbb{F}_q$ depends on the roots of the characteristic polynomial of the Frobenius endomorphism of $A/\mathbb{F}_q$.

To be more concrete, assume that $q=p \geq 5$ is a prime, that $\ell >p$, and that our abelian variety is a supersingular elliptic curve which we denote by $E/\mathbb{F}_p$. In this case, the polynomial $P_{\phi}(x)$ is a quadratic polynomial with roots $\alpha_1$ and $\alpha_2$. Since $E/\mathbb{F}_p$ is supersingular, we have that $\# E(\mathbb{F}_p)=p+1$ and, hence, using Equation $($\ref{equationpointsfinitefields}$)$ (or \cite[Theorem V.2.3.1]{aec}) we find that $\alpha_1=-\alpha_2$. This implies, by \cite[Theorem V.2.3.1]{aec}, that 
\begin{align*} \#E(\mathbb{F}_{p^{f\ell^e}})=
    \begin{cases}
        p^{f\ell^e}+1, & \text{if } f \text{ is odd} \\
        (p^{\frac{f\ell^e}{2}}-(-1)^{\frac{f\ell^e}{2}})^2, & \text{if } f \text{ is even}
    \end{cases}
\end{align*}

We now show that in this case, whether $\mathbb{F}_p(E[\ell^{\infty}])$ is contained in the maximal pro-$\ell$ extension of $\mathbb{F}_p$ depends on $f$. It follows from the Weil pairing, see \cite[Proposition 1.11]{masterthesis} for a self-contained proof, that if $\ell$ divides $\# E(\mathbb{F}_{p^{f\ell^e}})$, for some $e \geq 0$, then the extension $\mathbb{F}_p(E[\ell])/\mathbb{F}_{q^f}$ has degree equal to a power of $\ell$. Therefore, using a similar argument as in the second paragraph of the proof of Proposition \ref{proponepoint} we find that if $\ell$ divides $\# E(\mathbb{F}_{p^{f\ell^e}})$, then $\mathbb{F}_p(E[\ell^{\infty}])$ is contained in the maximal pro-$\ell$ extension of $\mathbb{F}_p$.

Assume that $f$ is even. We will show that if $f \equiv 0 \: (\text{mod } 4 ) $, then $\ell$ does not divide $\# E(\mathbb{F}_{p^{f\ell^e}})$ and that if $f \equiv 2 \: (\text{mod } 4 ) $, then $\ell$ does divide $\# E(\mathbb{F}_{p^{f\ell^e}})$. Since $f$ is even, we see that $\#E(\mathbb{F}_{p^{f\ell^e}})  = (p^{\frac{f\ell^e}{2}}-(-1)^{\frac{f\ell^e}{2}})^2$. Therefore, we have that$$\#E(\mathbb{F}_{p^{f\ell^e}}) \equiv ((p^{\ell^e})^{\frac{f}{2}}-((-1)^{\ell^e})^{\frac{f}{2}})^2 \equiv (p^\frac{f}{2}-(-1)^\frac{f}{2})^2  \; (\text{mod } \ell).$$ On the other hand, since $f$ is the smallest integer such that $p^f \equiv 1 \: (\text{mod } \ell ) $, we find that $p^\frac{f}{2} \equiv -1 \: (\text{mod } \ell ) $. Consequently, if $f \equiv 2 \: (\text{mod } 4 ) $, then $\#E(\mathbb{F}_{p^{f\ell^e}}) \equiv (-1+1)^2 \equiv 0 \: (\text{mod } \ell ) $. On the other hand, if $f \equiv 0 \: (\text{mod } 4 ) $, then $\#E(\mathbb{F}_{p^{f\ell^e}}) \equiv (-1-1)^2 \equiv 4 \: (\text{mod } \ell ) $ and, hence, $\ell$ does not divide $\#E(\mathbb{F}_{p^{f\ell^e}})$.

Finally we show that if $f$ is odd, then $\ell$ does not divide $\# E(\mathbb{F}_{p^{f\ell^e}})$. Indeed, since $f$ is odd, we see that $\#E(\mathbb{F}_{p^{f\ell^e}})  = p^{f\ell^e}+1$. Therefore, we have that $$\#E(\mathbb{F}_{p^{f\ell^e}}) \equiv (p^{\ell^e})^{f}+1 \equiv p^f+1 \equiv 2 \; (\text{mod } \ell),$$ where the last congruence is true because $p^f \equiv 1 \: (\text{mod } \ell ) $.

\begin{remark}
    For this remark we keep the same notation as in the previous paragraphs. A natural question is whether for a fixed prime $p$, one can find infinitely many primes $\ell$ for which $f \equiv 1 \: (\text{mod } 4 ) $ (or $f \equiv 2 \: (\text{mod } 4 ) $, or $f \equiv 0 \: (\text{mod } 4 ) $). The answer to this question is positive and it follows from \cite{lombardoperucca} and \cite{perucca}. More precisely, the density of primes $p$ such that $f \equiv 1 \: (\text{mod } 4 ) $,   $f \equiv 2 \: (\text{mod } 4 ) $ and $f \equiv 0 \: (\text{mod } 4 ) $) are the quantities $D_2(q,0)$, $D_2(q,1)$, and $1-D_2(q,0)-D_2(q,1)$ in \cite{perucca}, respectively (note that in \cite{perucca} the letter $p$ is used for what we call $\ell$ here). Moreover, it follows from \cite[Table 1]{perucca} that all these densities are positive (see also \cite[Remark 14]{perucca} and \cite[Theorem 17]{perucca}). Finally, we note that \cite[Theorem 1]{lombardoperucca} is more general version of these results.
\end{remark}

\section{a possible generalization}

Let $k$ be a perfect field of characteristic $p \geq 0$ and let $C/k$ be a smooth, projective, and geometrically connected curve with function field $K=k(C)$. One may wonder whether we can replace the extensions $K(\mu_{\ell})/K$ with more general (possibly ramified) extensions in Conjecture \ref{conjecturefunctionfields}. More precisely, for every prime $\ell \neq p$ let $M_{\ell}/K$ be a finite extension and for each $M_{\ell}$ consider the maximal pro-$\ell$ extension $M_{\ell}^{\text{pro}-\ell}$ of $M_{\ell}$. One may now wonder for which $M_{\ell}$ would the statement of Conjecture \ref{conjecturefunctionfields} still make sense after replacing $\Tilde{K_{\ell}'}$ with $M_{\ell}^{\text{pro}-\ell}$. 

As an example, we show below that if for every $\ell$ we take $M_{\ell}/K$ to be a Galois extension of degree dividing $\ell-1$, then we still have that the variant of Conjecture \ref{conjecturefunctionfields} is true for elliptic curves.

\begin{theorem}
Let $k$ be a perfect field and let $C/k$ be a smooth, projective, and geometrically connected curve of genus $g$ with function field $K=k(C)$. Let $E/K$ be a non-isotrivial elliptic curve, let $\ell$ be a prime number, and let $L/K$ be a Galois extension of degree dividing $\ell-1$. If the degree of the extension $K(E[\ell])/L$ is a power of $\ell$, then $\ell \leq 49 \max \{ 1, g\}$.
\end{theorem}
\begin{proof}

Let $K$ be as in the theorem and let $E/K$ be a non-isotrivial elliptic curve. We will show that if the degree of the extension $K(E[\ell])/L$ is a power of $\ell$, then $\ell \leq 49 \max \{ 1, g\}$. The main tool in our proof is the following lemma.

\begin{lemma}\label{mainlemma}
The $\mathbb{F}_{\ell}$-vector space $E[\ell]$ has a cyclic $\text{Gal}(K^{\text{sep}}/K)$-stable subgroup.
\end{lemma}
\begin{proof}[{\it Proof of the lemma}]
The strategy for the proof of the lemma is inspired by the arguments of the proof of \cite[Lemma 3]{rt08}. We include all the relevant details in order to make our proof self-contained.

Let $G=\text{Gal}(K(E[\ell])/K)$, let $N=\text{Gal}(K(E[\ell])/L)$, and let $\Delta=G/N=\text{Gal}(L/K)$. By assumption, the group $N$ is a finite $\ell$-group. Therefore, each of the orbits of $E[\ell]$ under $N$ must have order equal to a power of $\ell$. This implies that the subspace  $E[\ell]^N$ of fixed points of $E[\ell]$ under the action of $N$ is not equal to $\{ 0 \}$. Indeed, if $E[\ell]^N = \{ 0 \}$, then we would have that $\ell$ divides $\#E[\ell]-1$ which is not possible. 

If $\Delta=G/N$ is trivial, which occurs when $L=K$, then we have that $E[\ell]^G=E[\ell]^N \neq   \{ 0 \}$. Therefore, either the action of $G$ on $E[\ell]$ is trivial or $E[\ell]^G$ is a subspace of dimension $1$. Hence, either the action of $\text{Gal}(\widebar{K}/K)$ on $E[\ell]$ is trivial or $E[\ell]^{\text{Gal}(K^{\text{sep}}/K)}=E[\ell]^G$ is a subspace of dimension $1$. This proves our lemma in the case where $\Delta$ is trivial. Therefore, we can assume from now on that $\Delta \neq \{ 0 \}$.

Assume from now on that $\Delta$ is not trivial. Since $N$ is a normal subgroup of $G$, it follows that $E[\ell]^N $ is $G$-stable (as a set). Consequently, there is a well defined action of $\Delta=G/N$ on $E[\ell]^N $. Fix a basis of $E[\ell]^N $ and consider the associated representation $\rho: \Delta \longrightarrow \text{GL}(E[\ell]^N)$. If $\delta$ is a generator of $\Delta$, since $\delta^{\ell-1}=1$, we have that $\rho(\delta)^{\ell-1}=I$. Therefore, the minimal polynomial of $\rho(\delta)$ splits completely over $\mathbb{F}_{\ell}$ and, hence, the matrix $\rho(\delta)$ has an eigenvalue $\lambda$ in $\mathbb{F}_{\ell}^{\times}$. Let $W$ be the subspace generated by an eigenvector with respect to $\lambda$. Then $\Delta$ fixes $W$ (as a set), and since $W \subseteq E[\ell]^N $, we see that $W$ is $G$-stable. This proves that $W$ is a $\text{Gal}(K^{\text{sep}}/K)$-stable subgroup of $E[\ell]$. This proves the lemma.
\end{proof}

We are now ready to complete our proof. By Lemma \ref{mainlemma} we have that there exists a cyclic, $\text{Gal}(K^{\text{sep}}/K)$-stable subgroup of $E$ which has order $\ell$. However, since $E/K$ is non-isotrivial and $\ell \neq p$ by assumption, it follows from \cite[Proposition 6.5]{griffonpazuki} that $\ell \leq 49\max \{ 1, g\}$, where $g$ is the genus of the curve $C/k$. This proves our theorem.
\end{proof}

\begin{corollary}
    Let $k$ be a perfect field and let $C/k$ be a smooth, projective, and geometrically connected curve of genus $g$ with function field $K=k(C)$. For every prime $\ell$ let $M_{\ell}/K$ be a Galois extension of degree dividing $\ell-1$ and consider the set $\mathscr{A}_{M_{\ell}}'(K,1, \ell)$ which consists of all non-isotrivial elliptic curves $E/K$ such that $K(E[\ell^{\infty}]) \subseteq M_{\ell}^{\text{pro}-\ell}.$  Then we have that $\mathscr{A}_{M_{\ell}}'(K,1, \ell) = \emptyset$ for $\ell > 49 \max \{ 1, g\}$. 
\end{corollary}

\bibliographystyle{plain}
\bibliography{bibliography.bib}

\end{document}